\documentclass[11pt,a4paper]{article}
 
\usepackage{amsmath}
\usepackage{amsfonts}

\newtheorem{theorem}{\large Theorem} 
\newtheorem{proposition}[theorem]   {\large Proposition}

\newtheorem{lemma}[theorem]{\large Lemma}
\newtheorem{remark}[theorem]{\large Remark}

\def\EE{{\Bbb E}}

\def\ZZ{{\Bbb Z}}

\def\sqr#1#2{{\vcenter{\vbox{\hrule height.#2pt
            \hbox{\vrule width.#2pt height#1pt \kern#1pt
            \vrule width.#2pt}
            \hrule height.#2pt}}}}

\def\abstract#1{
\goodbreak\bigskip\bigskip\noindent{\bf Abstract.}
{\it\narrower  #1\smallskip} \medskip}
\def\ni{{\noindent}}
\def\sobre#1#2{\lower 1ex \hbox{ $#1 \atop #2 $ } }
\def\'#1{{\if #1i{\accent"13\i}\else {\accent"13 #1}\fi}}
\def\bajo#1#2{\raise 1ex \hbox{ $#1 \atop #2 $ } }

\title{This is the title of the paper}
\author{Firstname Lastname, second author}

\begin{document} 

\centerline{\bf $R-$positivity of matrices and Hamiltonians }
\centerline{\bf on nearest neighbors trajectories}

\bigskip

\centerline{\bf Jorge Littin}
\centerline{ DIM and CMM Universidad de Chile, UMI 2807 CNRS}
\centerline{ Casilla 170-3, Santiago, Chile,
e-mail: jlittin@dim.uchile.cl}

\medskip

\centerline{\bf Servet Mart\'{\i}nez}
\centerline{DIM and CMM Universidad de Chile, UMI 2807 CNRS}
\centerline{Casilla 170-3, Santiago, Chile,
e-mail: smartine@dim.uchile.cl}

\abstract{ We revisit the $R-$positivity of nearest
neighbors matrices on ${\ZZ_+}$ and the Gibbs 
measures on the set of nearest neighbors trajectories on ${\ZZ_+}$
whose Hamiltonians award either visits to sites a or visits to edges. 
We give conditions that guarantee the $R-$positivity or 
equivalently the existence of the 
infinite volume Gibbs measure, and we show geometrical recurrence of 
the associated Markov chain. In this work we generalize and 
sharpen results obtained in \cite{fm} and \cite{hmsm}.} 

\medskip

\ni{\bf Keywords}: Nearest neighbors, $R-$ positivity, Gibbs 
measures.  

\medskip

\ni{\bf AMS 2010 subject classification}: 60J10, 60K35, 82B.

\medskip

\section{$R-$ positivity and main result}

Let $Q=(q_{x,y}:x,y\in\ZZ_+)$ be a nearest neighbors matrix on $\ZZ_+$, 
i.e.
\begin{equation}
\label{eq1.9}
q_{x,y}=0\hbox{ if }|x-y|\neq 1\hbox{ and } q_{x,y}>0\hbox{ if }
|x-y|=1.
\end{equation}
From irreducibility we get that
$$
R(Q)=\left(\limsup\limits_{N\to\infty}
\left(q^{(2N)}_{x,x}\right)^{1/2N}\right)^{-1}\,,
$$
is a common convergence radius, i.e. it is independent of 
$x\in\ZZ_+$. In this work we will assume $R(Q)>0$. Let us put $R=R(Q)$. 

\medskip

The matrix $Q$ is said to be $R$-recurrent if $\sum\limits_{n\ge 0}
q^{(2N)}_{x,x}R^n=\infty$, otherwise it is called $R$-transient.
Consider the continued fraction
$$
H(Q,R):={\displaystyle {q_{0,1} q_{1,0}/R^2
\over{\displaystyle 1-{q_{1,2} q_{2,1}/R^2\over
{\displaystyle \cdots \atop{ \displaystyle 1 -
{q_{x,x+1} q_{x+1,x}/R^2 \over \cdots }}}}}}}\,.
$$
For the class of matrices $Q$ of type (\ref{eq1.9}),
it was shown in (\cite{fm2})
that $Q$ is $R-$recurrent if and only if $H(Q,R)=1$
and it is $R-$transient when $H(Q,R)<1$ (the proof uses 
strongly Theorem 11.2 of Wall in \cite{w}). 
We have that there exists a solution to the eigenvector problem:
$$
Q\vec f=R^{-1}\vec f\hbox{ with }\vec f=(f_x:x\in\ZZ_+)>0\,.
$$
(For general positive matrices the existence of a solution is 
only guaranteed for $R-$recurrent matrices).
Note that the matrix $P^{(R)}=(p_{x,y}:x,y\in\ZZ_+)$ defined by
$$
p_{x,y}=R q_{x,y}\frac{f_y}{f_x}\quad x,y\in\ZZ_+ ,
$$
is an stochastic matrix of a birth and death chain $X^{Q}=(X^{Q}_n: 
n\ge 0)$
reflected at $0$. The matrix $Q$ is said to be $R$-positive
recurrent if $P^{(R)}$ is a positive recurrent Markov chain.
(For the previous definitions and results see Vere-Jones
in \cite{v1} and \cite{v2}).

\medskip

For $x\ge 0$ we denote by $\tau_x:=\inf\{n>0: X^{Q}_n=x\}$
the return time to $x$, where as usual we put $\tau_x=\infty$
when $X^{Q}_n\neq x$ for all $n>0$.

\medskip

Denote $J^{[m]}=\{x\in \ZZ_+: x\ge m\}$. Let us consider the family of 
matrices 
\begin{equation}
\label{eqseq}
Q^{[m]}=(q_{x,y}: x,y\in J^{[m]})\,, \; m\ge 0\,. 
\end{equation}
With this notation $Q^{[0]}=Q$ the original matrix. 
Denote by
$R^{[m]}:=R(Q^{[m]})$ the convergence radius of $Q^{[m]}$, and
by $X^{[m]}:=X^{Q^{[m]}}=(X^{[m]}_n: n\ge 0)$ the associated 
stochastic matrix.
 By definition the sequence $(R^{[m]}: m\ge 0)$ is increasing:
$$
\forall m\ge 0\,:\quad\quad R^{[m]}\leq R^{[m+1]}\,,
$$
so we can only have two situations:

\medskip

\noindent $\bullet$ The sequence is constant: $R^{[0]}=R^{[m]}$,
$\forall m\geq 0$;

\medskip

\noindent $\bullet$  The sequence has a gap, so there exists
some $m\ge 0$ such that $R^{[m]}<R^{[m+1]}$.

\medskip

Our main result is:

\begin{theorem}
\label{theo1.1}
Let $Q$ matrix of type (\ref{eq1.9}). Assume $R(Q)>0$.  
If there exists a gap, 
$R^{[m]}<R^{[m+1]}$ for some $m\ge 0$, then the matrix
$Q$ is $R-$positive. Moreover the associated birth and death chain 
$X^Q$ is geometrically recurrent: 
$$
\forall \, 1< \theta \le {R^{[m+1]}}^2 /{R^{[m]}}^2
\,:
\quad \;
\EE_x(\theta^{\tau_x})<\infty\,.
$$

When there does not exists a gap, $R^{[0]}=R^{[m]}$ for all 
$m\ge 0$, then the matrix $Q^{[1]}$ is $R-$transient.
\end{theorem}

We point our that this result generalizes the result on \cite{hmsm}
in two directions: the matrix $Q$ is not 
necessarily substochastic as is the case in \cite{hmsm}
and on the other hand we can sharpen $R-$positivity to
geometrically recurrence of the associated Markov chain.
The tools of this work are close to those used in \cite{fm}.

\section{\bf Hamiltonians and Gibbs measures}

Let us put our result in the context of one-dimensional 
Gibbs measures.
We consider the space of trajectories $\Omega$ of a
nearest neighbors non negative random walk,
$$
\Omega =\{w\in\ZZ^\ZZ_+ :
|w_i-w_{i+1}|=1\ \forall i\in\ZZ\}.
$$
For $i\le j$ we put $w[i,j]=(w_i,\dots,w_j)$ and 
$\Omega [i,j]=\{\sigma = w[i,j]: w\in\ZZ^\ZZ_+ \}$.
Let $\vec\alpha=(\alpha_x:x\in\ZZ_+)$,
$\vec b=(b_x:x\in\ZZ_+)$, $\vec c=(c_x:x\in\ZZ_+)$
be fixed sequences (the rewards).
For $x,y \in\ZZ_+$ and a block $\sigma\in\Omega[i,j]$ we denote by
$$
{\cal N}_x(\sigma)=\sum\limits^j_{k=i}\delta(\sigma_k,x)
\hbox{ and }
{\cal N}_{x,y}(\sigma)=
\sum\limits^{j-1}_{k=i}\delta(\sigma_k\sigma_{k+1},xy),
$$
the number of times $\sigma$ visits $x$ and
the number of times $\sigma$ passes through the edge $xy$
respectively, this last quantity vanishing if $|x-y|\neq 1$.
The Hamiltonians that respectively award the number of visits to 
the sites or to the edges, are the following ones
on the interval $[i,j]$. For $w\in\Omega$, 
\begin{eqnarray*}
H_{[i,j]}^{\vec\alpha}(w)&=&\sum\limits_{x\in\ZZ_+} \alpha_x{\cal N}_x
(w[i,j+1]) \hbox{ and }\\
H_{[i,j]}^{\vec b,\vec c}(w)&=& \sum\limits_{x\in\ZZ_+}
 \Big(b_x{\cal N}_{x,x+1}(w[i-1,j+1]) +c_x {\cal N}_{x+1,x}(w[i-1,j+1]) 
\Big) \,.
\end{eqnarray*}
 For the Hamiltonians $H=H^{\vec\alpha}$ and $H=H^{\vec b,\vec c}$,
the probability measures associated to them are
(see \cite{g} Definition 2.9)
$$
\mu_{[i,j]}^H(w)(\sigma)=(Z_{[i,j]}(w))^{-1}
\sum\limits_{\displaystyle w'[i-1,j+1]: w'[k,\ell]=\sigma
\atop{\displaystyle w'_{i-1}=w_{i-1}, w'_{j+1}=w_{j+1}}}\!\!\!
e^{H_{[i,j]}(w')}\,,\;\; \sigma\in\Omega[k,\ell],
$$
where $i<k, \ell<j$ and $Z_{[i,j]}(w)=
\sum\limits_{w'[i-1,j+1]:\atop{w'_{i-1}=w_{i-1},w'_{j+1}=w_{j+1}}}
e^{H_{[i,j]}(w')}$ is
the partition function of $w\in\Omega$ in $[i,j]$. 
We ask for the conditions on $\vec \alpha$, or in $\vec b$ and $\vec c$, 
such that the Hamiltonians $H^{\vec \alpha}$, or $H^{\vec b,\vec c}$,
define translational invariant Gibbs measures. 

\medskip

These Hamiltonians are related by the following
equalities shown in \cite{fm}: when  the sequences $\vec \alpha$, $\vec 
b$, $\vec c$ verify $\alpha_x+\alpha_{x+1}= b_x+c_x$ for $x\in\ZZ_+$
then there exists a real function $\gamma(n,m,p)$ defined in $\ZZ^3_+$ 
such that
$$
H^{\vec b, \vec c} _{[i,j]}(w)= 
H^{\vec\alpha}_{[i,j]}(w)+ \gamma(w_{i-1},w_{j+1},j-i) .
$$
In particular this relation implies
$$
\mu^{H^{\vec b, \vec c}}_{[i,j]}(w)(\sigma) =\mu^{H^{\vec \alpha}}
_{[i,j]}(w)(\sigma)\hbox{ for any }\sigma\in \Omega[k,\ell],
\hbox{ with } [k,\ell] \subseteq [i+1,j-1] .
$$
 From this result we can restrict ourselves to analyze when the 
Hamiltonian $H^{\vec b, \vec c}$ defines an infinite volume Gibbs 
measure because we can always fit  $\vec b$ and $\vec c$ to have 
$\alpha_x+\alpha_{x+1}= b_x+c_x$ for $x\in\ZZ_+$. Associated to the 
Hamiltonian $H^{\vec b, \vec c}$ is the
transfer matrix $Q=(q_{x,y}: x,y\in \ZZ_+)$ of type (\ref{eq1.9}) 
where $q_{x,x+1} = e^{b_x}$ and $q_{x+1,x} = e^{c_x}$ for
$x\ge 0$; and $q_{x,y}=0$ otherwise. We have,
$$
\mu^{H^{\vec b, \vec c}}_N(w)(\sigma)=
{Q^{N-k+1}(w_{-(N+1)},\sigma_{-k})\prod\limits^{k-1}_{i=-k}Q(\sigma_i,
\sigma_{i+1})
Q^{N-k+1}(\sigma_k,w_{N+1})\over Q^{2N+2}(w_{-(N+1)},w_{N+1})}\,.
$$
 From Theorem 1 in Kesten \cite{k}
for strictly positive matrices and extended in Theorem C in 
\cite{gu} for irreducible matrices, we have
that there exists a unique translational invariant Gibbs
state for the Hamiltonian $H^{\vec  b,\vec c}$ if and
only if $Q$ is a $R$-positive matrix.

\medskip

Therefore Theorem \ref{theo1.1} 
give a sufficient condition for the general case of
$\vec b$ and $\vec c$ in order that there exists a unique translational 
invariant Gibbs state for the Hamiltonian $H^{\vec  b,\vec c}$. This 
result goes beyond the cases analyzed in \cite{fm}.  
We recall that when $\vec b+ \vec c$ is constant for $x$ sufficiently
large (that is the sequence is ultimately constant) in 
Theorems 1.2 and Theorem 1.4 in \cite{fm}, 
it was given necessary and sufficient explicit conditions for the 
existence of a Gibbs measure in terms of $\vec b+ \vec c$. 
For awards on sites and when $\alpha_{x+2}\le \alpha_x$
for $x$ sufficiently large, in \cite{fm} there was supplied sufficient 
conditions for the existence of a Gibbs measure.
All these conditions were written in terms of continued fractions.
We finally point out that in \cite{fm} it was also discussed the 
relations between the results on Hamiltonians awarding the visits to 
sites with the entropic repulsion of a wall and with the SOS model.

\section{ Proof of the main result }

Let $\vec f=(f_x:x\in\ZZ_+)$. We consider the general eigenvalue problem
for matrices of type (\ref{eq1.9}):
\begin{equation}
\label{eq2.1}
Q\vec f=r^{-1} \vec f\hbox{ for } r>0,\, \vec f>0\,.
\end{equation}
For $r\in [R,\infty)$ there is a unique, up to a
homothetic transformation, $\vec f >0$ verifying (\ref{eq2.1}).
Moreover if for some $r>0$ there exists a solution to (\ref{eq2.1})
then necessarily $r\in [R,\infty)$ (see \cite{fm}).

\medskip

Let $r\in [R,\infty)$. The matrix 
$P^{(r)}=(p_{x,y}:x,y\in\ZZ_+)$ defined by
\begin{equation}
\label{eq2.2}
p_{x,y}=r {f_y\over f_x} q_{x,y}\hbox{ for }x,y\in\ZZ_+^*
\end{equation}
is a stochastic matrix of a birth-death chain reflected
at $0$. We have that $P^{(r)}$ is transient for all $r\in (R, \infty)$. 

\medskip

Let us put
$$
\omega_x\doteq p_{x,x+1}=r {f_{x+1}\over f_x}q_{x,x+1}. 
$$
 From (\ref{eq2.2}) the sequence $(w_x:x\in\ZZ_+)$ verifies the equation:
\begin{equation}
\label{eq2.4}
\omega_0=1\hbox{ and } \omega_{x+1}= 1-{r^2 q_{x,x+1} q_{x+1,x}\over 
\omega_x}
\hbox{ for } x\in\ZZ_+ .
\end{equation}
Conversely it is direct to prove that if the sequence
$\vec \omega=(\omega_x:x\in\ZZ_+)$
given by the evolution (\ref{eq2.4}) verifies
$\vec \omega>0$, then $\vec f$ defined by $f_0>0$ and 
$f_{x+1}=f_0 r^{x+1}\prod\limits^x_{y=0}{\omega_y \over q_{y,y+1}}$ for
$x\in\ZZ_+$, verifies (\ref{eq2.1}).

\medskip

At this point it is convenient to introduce some new notation and a
definition. First,
for $a>0$ we consider the following continuous and onto
strictly
increasing function $\varphi_a:(0,\infty]\to (-\infty ,1]$,
$$
\varphi_a(\omega) =1-\frac{a}{\omega}.
$$

\ni{\bf Definition 2.1}.
Let $\vec a=(a_x>0:x\in\ZZ_+)>0$ be a strictly positive fixed
sequence.
It is said to be allowed
if it verifies
$$
\forall x\in\ZZ_+\,:\quad \varphi_{a_x}\circ \dots \circ \varphi_{a_0}(1)>0
\,.  
$$
Observe that
the inverse $\varphi^{-1}_a(\omega)=\frac{a}{1-\omega}$
satisfies analogous properties as $\varphi_a$.
Also from the definition we get
$$
\hbox{if } \omega>0 \hbox{ and }\varphi_a(\omega)>0
\hbox{ then }\varphi_a (\omega)\in(0,1).
$$

The first part $(i)$ of the next result was already proven in
\cite{fmp} and the parts $(ii)$ and $(iii)$
were shown in \cite{fm}.

\begin{lemma}
\label{lemma2.2}
Let ${\vec a}>0$ be a strictly positive sequence.

\medskip

\noindent $(i)$  ${\vec a}$ is allowed
if and only if it verifies
$$
\forall x\in\ZZ_+  ,\ \forall y\ge x: \varphi^{-1}_{a_x}
\circ \dots\circ\varphi^{-1}_{a_y}(0)<1 . \eqno(2.7)
$$

\noindent $(ii)$ Let $\vec d>0$  be a strictly positive sequence, then
$$
\vec d\le \vec a\hbox{ and }\vec a\hbox{ is allowed
implies }\vec d \hbox{ is allowed }.
$$

\noindent $(iii)$ Let us denote
$$
s\vec a=(sa_x:x\in\ZZ_+)\hbox{ for }s>0\hbox{ and }
{\cal I}(\vec a)=\{s>0: s\vec a \hbox{ is allowed }\}.
$$
Then ${\cal I}(\vec a)=\emptyset$ or
${\cal I}(\vec a)=(0,s^*]$ for some $s^*\in (0,\infty)$.
\end{lemma}

\medskip

As a Corollary to this Lemma and by using \cite{fmp} and \cite{fm},
we find that when the sequence $\vec a$ verifies $a_x=q_{x,x+1} 
q_{x+1,x}$ then $s^*=R(Q)^{2}$ (here $q_{x,y}$ are the coefficients of 
the matrix $Q$ of type (\ref{eq1.9})). In this case
${\cal I}(\vec a)\neq \emptyset$ if and only if $R(Q)>0$. 

\medskip

Assume $\vec a$ is allowed then $\varphi^{-1}_{a_{y+1}}(0)\in(0,1)$
and by the increasing property we get
$$
h_{\vec a}(x,y)=\varphi_{a_x}^{-1}\circ \dots\circ
\varphi^{-1}_{a_y}(0)
<\varphi^{-1}_{a_x}\circ \dots\circ \varphi^{-1}_{a_{y+1}}(0)=
h_{\vec a}(x,y+1).
$$
i.e. the sequence $h_{\vec a}(x,y)$ is strictly increasing
in $y\in\ZZ_+^*$.
Then the following limit exists and verifies:
$$
h_{\vec a}(x,\infty)=\lim\limits_{y\nearrow\infty}
h_{\vec a}(x,y)\le 1\,.
$$
Observe that
$$h_{\vec a}(x,y)={\displaystyle a_x
\over{\displaystyle 1-
{\displaystyle a_{x+1}\over 1-{\displaystyle a_{x+2}\over
{\displaystyle  \cdots \atop{\displaystyle 1-a_y}}}}}}\ ,$$
then $h_{\vec a}(x,\infty)$ is a continued fraction.

\medskip

Recall the notation (2.9), $\omega_0=1$, $\omega_{x+1}=\varphi_
{a_x}\circ \dots \circ \varphi_{a_0}(1)$ for $x\in\ZZ_+$.
Assume $\vec a$ is a fixed sequence. We put,
$$
\forall s\in (0,s^*] \,: \quad\; 
h(s;x,y):=h_{s \vec a}(x,y) \hbox{ and }
h(s;x,\infty):=h_{s\vec a}(x,\infty)\,.
$$

Recall the notation $J^{[m]}=\{x\in \ZZ_+: x\ge m\}$.
Let us consider the family of shifted sequences 
$$
{\vec a}^{[m]}=(a_x: x\in J^{[m]})\,, \;\; m\ge 0 \,,
$$
and the associated values, 
$$
s^{*[m]}=\sup \{s \in \mathcal{I}({\vec a}^{[m]})\}\,, \;\;
m \geq 0\,. 
$$
From definition this sequence is increasing: $s^{*[m]}\leq s^{*[m+1]}$
for all $m\ge 0$ and so we can only have
that there are two possibilities: the sequence is constant i.e. 
$s^{*[0]}=s^{*[m]}$ 
$\forall m\geq 0$; or the sequence has a gap, that is there exists 
some $m\ge 0$ such that $s^{*[m]}<s^{*[m+1]}$.

\begin{lemma}
\label{lemmacaso1}
If $s^{*[m]}<s^{*[m+1]}$ then 
$h(s^{*[m]};m,\infty)=1>h(s^{*[m+1]};m,\infty)$.
\end{lemma}

\begin{proof}
 From $s^{*[m]}<s^{*[m+1]}$ we get
$h(s^{*[m]};m,\infty)\leq 1$ and $h(s^{*[m+1]};m,\infty)> 1$. 
The last relation follows because $h(s^{*[m+1]};m,\infty)\le 1$
would imply $s^{*[m]}=s^{*[m+1]}$, a contradiction. 
So we only left to prove that under the 
assumption we have $h(s^{*[m]};m,\infty)=1$.

\medskip

We claim that: 

\smallskip

\noindent $\bullet$ $h(s;m,\infty)<1$ then   
$h(s;m,\infty)$ is an increasing and continuous function in 
$s\in \Delta_m:=[s^{*[m]},s^{*[m+1]}]$.

\medskip

Let us show the increasing part of the claim. Since
$0<h(s;m,y)<1$  then it is increasing in $y\ge m$. On the other hand 
for every fixed $y$ we have that for all $s<s'$ with $s,s'\in \Delta_m$, 
we have $0<h(s;m,y)<h(s';m,y)<1$. Then by taking 
$\lim_{y \rightarrow \infty}$ in this inequality we conclude
$h(s;m,\infty) \leq h(s';m,\infty)$.

\medskip

Now, let show the continuity part of the claim. Under the assumption,  
for $s<s'$ with $s,s'\in \Delta_m$, we get that the inequality 
\begin{eqnarray*}
|h(s';m,\infty)\!-\!h(s;m,\infty)| &\!\leq\!& 
|h(s';m,\infty)\!-\!h(s';m,y)|
\!+\!|h(r;m,y)\!-\!h(s';m,y)|\\
&{}& \,+|h(s;m,\infty)\!-\!h(s;m,y)|\,,
\end{eqnarray*}
is verified for all fixed $y\ge m$. 

\medskip

Let $\epsilon>0$ be fixed. We can find $y\ge m$ sufficiently big such that
$|h(s';m,\infty)-h(s';m,y)|<\epsilon/3$ and 
$|h(s;m,\infty)-h(s;m,y)|<\epsilon/3$. Let us fix one of such $y$.
Since $h(s;x,y)$ is continuous when
$0<h(r;x,y)<1$, there exists 
$\delta>0$ such that  
$0<|s-s'|<\delta$, $s,s'\in \Delta_m$, implies
$|h(s;x,y)-h(s';x,y)|<\epsilon/3$. 
The continuity is verified. 

\medskip

From the claim it results that $h(s^{*[m]};m,\infty)=1$. Indeed,
in the contrary we should have $h(s^{*[m]};m,\infty)<1$ and 
$h(s^{*[m+1]};0,\infty)>1$. Since $h(s;m,\infty)$ is increasing and 
continuous there would exist 
$s'\in (s^{*[m]},s^{*[m+1]})$ such that $h(s';m,\infty)=1$. 
Since $h(s';m,\infty)=\lim\limits_{y \to \infty}
a_0 s'/(1-h(s';1,y)$, 
we should get $0<h(s';m,y)<1$ $\forall y\geq m$ 
and so $s'>s^{*[m]}$ is such that $s'{\vec a}$ is allowed, 
which contradicts the maximality property satisfied by 
$s^{*[m]}$. $\Box$
\end{proof}

\medskip

Let us study the $R$-positivity in Theorem \ref{theo1.1},
so we are under the hypothesis $R(Q)>0$.
We fix the sequence
$$
a_x=q_{x,x+1} q_{x+1,x}\,, \; x\ge 0\,.
$$
For the matrices $Q^{[m]}$ defined in
(\ref{eqseq}) we consider the sequences 
${\vec a}^{[m]}$ already defined and 
the reflected birth and death chain 
$X^{[m]}=(X_n^{[m]}: n\ge 0)$ taking values in $J^{[m]}$
with transition matrices,
\begin{equation}
\label{eq71}
\omega_m^{[m]}=1 \hbox{ and } 
\omega_{x+1}^{[m]}=1-\frac{{R^{[m]}}^2 a_{x}}{\omega_{x}^{[m]}}, 
\; x \in J^{[m]}\,,
\end{equation}
(see (\ref{eq2.4})).
The transition probabilities of the birth and death chain 
$X^{[m]}$ are $p^{[m]}_{x,x+1}= \omega^{[m]}_x$ and 
$p^{[m]}_{x+1,x}=1-\omega^{[m]}_{x+1}$. 
For $x\in J^{[m]}$ we denote by $\mathbb{P}_x^{[m]}$ the 
probability distribution
of the chain $X^{[m]}$ when it starts from $X^{[m]}_0=i$. 
The above construction is done for all $m\ge 0$.

\medskip

Let
$$
\xi_m:=\left(\frac{R^{[m]}}{R^{[m+1]}}\right)^2\,,
$$
From (\ref{eq71}) the following identity is verified:
\begin{equation}
\label{eq72}
\forall x \in J^{[m+1]}:\;
\omega_x^{[m]}(1-\omega^{[m]}_{x+1})=
\xi_m\, \omega^{[m+1]}_x(1-\omega^{[m+1]}_{x+1})\,.
\end{equation}

Now, for all $y \in J^{[m]}$ define
$$
\tau_y^{[m]}=\inf\{n>0: X^{[m]}_n=y\}\,,
$$
where as usual $\tau^{[m]}_y=\infty$ when $X^{[m]}_n \neq y$
for all $n>0$.  

\begin{proposition}
\label{prop11}
If $R^{[m]}<R^{[m+1]}$ the chain $X^{[m]}$ is positive
recurrent. Moreover it has exponential moment, 
\begin{equation}
\label{eq221}
\forall \, \theta\in (1,\xi_m^{-1}),\; \forall x,y\in J^{[m]}\,:
\quad\quad \mathbb{E}_{x}\left(\theta^{\tau^{[m]}_{y}}\right)<\infty\,.
\end{equation}
\end{proposition}

\begin{proof}
For all $k\geq 2$ it holds the following relation, where
we put $x_0=m+1=x_{2k}$:
\begin{eqnarray*}
\mathbb{P}^{[m]}_{m+1}\left(\tau^{[m]}_{m+1}= 2k\right)
&=&\sum_{x_1,...,x_{2k-1}>m+1,\, 
 |x_{j}-x_{j+1}|=1} p^{[m]}_{m+1,i_1}
p^{[m]}_{x_1,x_{2}}p^{[m]}_{x_2,x_3}...p^{[m]}_{x_{2k-1},m+1}\\
&=&\sum_{m+1<x_1,...x_{k-1} ,|x_j-x_{j+1}|=1} 
\prod_{j=0}^{k-1} 
\omega^{[m]}_{x_j}\left(1-\omega^{[m]}_{x_j+1}\right)\\
&=&\sum_{m+1<x_1,...x_{k-1},\, |x_j-x_{j+1}|=1} 
\left(\prod_{j=0}^{k-1} \xi_m
\omega^{[m+1]}_{x_j}\left(1-\omega^{[m+1]}_{x_{j}+1}\right)\right)\\
&=&\sum_{x_1,...,x_{2k-1}>1, |x_{j}-x_{j+1}|=1,}\xi_m^k 
p^{[m+1]}_{m+1,x_1}p^{[m+1]}_{x_1,x_{2}}p^{[m+1]}_{x_2,x_3}...
p^{[m+1]}_{x_{2k-1,m+1}}\\
&=&\xi_m^k 
\mathbb{P}^{[m+1]}_{m+1}\left(\tau_{m+1}^{[m+1]}=2k\right)\,.
\end{eqnarray*}
where in the third line we used equality (\ref{eq72}).
For $k=1$ we have:
\begin{eqnarray*}
\mathbb{P}^{[m]}_{m+1}\left(\tau^{[m]}_{m+1}=2\right)= 
(1-\omega^{[m]}_{m+1})+\mathbb{P}^{[m+1]}_{m+1} 
\left(\tau^{[m+1]}_{m+1}=2\right)\,.
\end{eqnarray*}

From the hypothesis $\xi_m=\left(\frac{R^{[m]}}{R^{[m+1]}}\right)^2<1$. 
For $1<\theta \leq \xi_m^{-1}$ we have:
\begin{eqnarray*}
&{}& \mathbb{E}^{[m]}_{m+1} \left(\theta^{\tau^{[m]}_{m+1}} \right)= 
\theta^2 \mathbb{P}^{[m]}_{m+1} \left(\tau^{[m]}_{m+1}=2 \right)+
\sum\limits_{k \geq 2} \theta^{2k}\, \mathbb{P}^{[m+1]}_{m+1} 
\left(\tau^{[m+1]}_{m+1}=2k\right)\\
&{}&\; =\theta^2 (1-\omega^{[m]}_{m+1})+ 
\sum_{k \geq 1}  \left(\theta \xi_m\right)^{2k} 
\mathbb{P}^{[m+1]}_{m+1} \left(\tau^{[m+1]}_{m+1}=2k\right)\\
&{}& \; \leq \theta^2(1-\omega_{m+1}^{[m]})+ \sum_{k \geq 1} 
\mathbb{P}^{[m+1]}_{m+1} 
\left(\tau^{[m+1]}_{m+1}=2k\right)\\
&{}& \; =\theta^2(1-\omega^{[m]}_{m+1})+\mathbb{P}^{[m+1]}_{m+1} 
\left(\tau^{[m+1]}_{m+1}<\infty\right)<\infty\,.
\end{eqnarray*}
We have shown that the chain $X^{[m]}$ verifies (\ref{eq221}) for 
$x=y=m+1$. By irreducibility this holds for all $x,y\ge m$.

\medskip

The $R^{[m]}$-positive recurrence follows directly by this fact. Indeed,
since there exists $m^*\ge m$ such that $\xi_m^{-y}\ge y$ for all $y\ge 
m^*$ we get
\begin{eqnarray*}
\mathbb{E}^{[m]}_{m+1}\left({\tau^{[m]}_{m+1}}\right)&=& 
\mathbb{E}^{[m]}_{m+1}\left({\tau^{[m]}_{m+1}} 
{\bf 1}_{\tau^{[m]}_{m+1} \leq m^*}\right)+
\mathbb{E}^{[m]}_{m+1}\left({\tau^{[m]}_{m+1}} 
{\bf 1}_{\tau^{[m]}_{m+1}> m^*} \right)\\
&\leq& m^*+\mathbb{E}^{[m]}_{m+1}\left(\xi_m^{-\tau^{[m]}_{m+1}} 
{\bf 1}_{\tau^{[m]}_{m+1} > m^*}\right)<\infty\,.
\end{eqnarray*}
$\Box$
\end{proof}

\begin{proposition}
\label{prop22}
If $R^{[m]}<R^{[m+1]}$ then for all $k=0,...,m $ we have 
$R^{[k]}<R^{[k+1]}$ and the chain $X^{[k]}$ is 
positive recurrent and has exponential moment.
\end{proposition}

\begin{proof}
From Proposition \ref{prop11} it suffices to show $R^{[k]}<R^{[k+1]}$
for all $k=0,...,m$. Let us do it by contradiction. 
Assume the property does not hold, 
then fix $j$ as the bigger $k$ smaller than $m$ where 
the strict inequality fails. 
So, we have $R^{[j]}=R^{[j+1]}<R^{[j+2]}$. This implies 
$h({R^{[j+1]}}^2;j+1,\infty)<1$, in fact in the contrary we should have
$h({R^{[j]}}^2;j,\infty)=h({R^{[j+1]}}^2;j,\infty)=\infty$ which is a 
contradiction.  
Now, from $R^{[j+1]}<R^{[j+2]}$ we get 
$h({R^{[j+2]}}^2;j+1,\infty)>1$. 

\medskip

By the same argument as the one used in Lemma \ref{lemmacaso1}
we will conclude that there exists 
$s\in ({R^{[j+1]}}^2,{R^{[j+2]}}^2)$ such that 
$h(s;j+1,\infty)=1$, that contradicts the maximality of $R^{[j+1]}$. 
We conclude $R^{[j]}<R^{[j+1]}$ and the property holds. $\Box$
\end{proof}

\begin{proposition}
\label{prop23}
If $R^{[0]}=R^{[m]}$ for all $m\ge 0$ then 
the chain $X^{[k]}$ is transient for all $k\ge 1$.
\end{proposition}

\begin{proof}
In this case we necessarily have $h({R^{[0]}}^2;x,\infty)<1$, 
$\forall x \geq 1$ showing the assertion. $\Box$
\end{proof}

\medskip

From Propositions \ref{prop22} and \ref{prop23}
it follows Theorem \ref{theo1.1}.

\begin{remark}
If $R^{[0]}=R^{[m]}$ for all $m\ge 0$ we are not able to classify 
completely $X^{[0]}$. We can only assert that
$h({R^{[0]}}^2;0,\infty)< 1$ implies that $X^{[0]}$ is transient, and 
if $h({R^{[0]}}^2;0,\infty)=1$ then $X^{[0]}$
is recurrent. But, we cannot state when it is null or positive 
recurrent.
\end{remark}

\begin{remark}
Assume the hypothesis of Proposition \ref{prop23}. Let
$0<s<s^{*[0]}$. Then $0<h(s;0,\infty)<1$. Take 
$a_{-1}=\frac{1-h(s;0,\infty)}{s}$ then the sequence 
$\vec{a}^{(-1)}=(a_{-1},a_0,...,a_n,...)$ verifies 
$h(s;-1,\infty)=\frac{a_{-1}s}{1-h(s;0,\infty)}=1$, $h(s;x,\infty)<1$, 
$\forall x\geq 1$. On the other hand for $s<s'<s^{*[0]}$,  
$h(s';-1,\infty)=\frac{a_{-1}s'}{1-h(s';0,\infty)}>1$. So,
$s=s^{*[-1]}<s^{*[0]}$, and the sequence 
$\vec{a}_{-1}$ has a gap and so the extension of the matrix $Q$
to $\{-1,0,...\}$ with $q_{-1,0}q_{0,-1}=a_{-1}$, is $R-$positive. 
\end{remark}

\medskip

\noindent {\bf Acknowledgments}. The authors thank the support
from Programa Basal-CMM CONICYT. J.L. is indebted to CONICYT
Ph.D. fellowship and S.M. to Guggenheim fellowship.

\end{document}